\documentclass[graybox,envcountsame,envcountsect]{svmult}

\usepackage{mathptmx}       
\usepackage{helvet}         
\usepackage{courier}        

\usepackage{makeidx}         
\usepackage{graphicx}        
\usepackage{multicol}        
\usepackage[bottom]{footmisc}
\usepackage{url}

\usepackage{amssymb,amsmath}
\smartqed

\makeindex             

\newcommand{\Hom}{\operatorname{Hom}}
\newcommand{\End}{\operatorname{End}}
\newcommand{\murev}{\mu_{\textrm{rev}}} 

\newcommand{\RR}{\mathbb{R}}

\newcommand{\Tree}{\mathbb{T}_n}
\newcommand{\vertex}{v}
\def\kk{\kappa}

\newcommand{\Ext}{\operatorname{Ext}}

\newcommand{\isom}{\simeq}
\newcommand{\sgn}{\mathrm{sign}}
\newcommand{\A}{\mathcal A}

\renewcommand{\S}{\mathcal S}

\begin{document}

\title*{Acyclic cluster algebras revisited}
\author{David Speyer \and Hugh Thomas}
%
\institute{David Speyer\at Department of Mathematics, 
University of Michigan, Ann Arbor, MI, USA \email{speyer@umich.edu}
\and Hugh Thomas\at Department of Mathematics and Statistics,
University of New Brunswick, Fredericton, NB, Canada
\email{hthomas@unb.ca}}

%
%
\maketitle
{\it
{Dedicated to Idun Reiten on the occasion of her seventieth birthday}
}

\bigskip\bigskip

\abstract{We describe a new way to relate an
acyclic, skew-symmetrizable cluster algebra to the representation theory
of a finite dimensional hereditary algebra.  
This approach is designed to explain the $c$-vectors of the cluster
algebra.  We obtain a necessary and sufficient combinatorial criterion for
a collection of vectors to be the $c$-vectors of some cluster in the
cluster algebra associated to a given skew-symmetrizable matrix.  
Our approach also 
yields a simple proof of the known result that the $c$-vectors 
of an acyclic cluster algebra are sign-coherent, from which
Nakanishi and Zelevinsky have showed that it is possible to deduce
in an elementary way several important facts about cluster algebras 
(specifically: Conjectures 1.1--1.4 of \cite{DWZ}).  
}

\section{Introduction}

Let $B^0$ be an acyclic skew-symmetrizable $n\times n$ integer matrix.  
Let $\widetilde B^0$
be the $2n\times n$ matrix whost top half is $B^0$ 
and whose bottom half is an $n\times n$ identity matrix.  

We consider an infinite $n$-ary tree $\Tree$, 
with each edge labelled by a number
from 1 to $n$, such that at each vertex, there is exactly one edge with
each label.  We label one vertex $\vertex_b$, and we associate the matrix 
$\widetilde B^0$ to it. 

There is an 
operation called \emph{matrix mutation} which plays a
fundamental role in the construction of cluster algebras.  (We recall
the definition in Section \ref{ca_background}.)  Using this definition, it is 
possible to associate a $2n\times n$ matrix to each vertex of
$\Tree$, so
that if two vertices are joined by an edge labelled $i$, the corresponding
matrices are related by matrix mutation in the $i$-th position.  

Let $\vertex\in \Tree$.
 We write $\widetilde B^\vertex$ for 
the associated $(2n\times n)$ 
$B$-matrix, and $B^\vertex$ for its top half.
The $c$-vectors for $\vertex$, denoted
$c_1^\vertex,\dots,c_n^\vertex$ are by definition the columns of the bottom half of 
$\widetilde B^\vertex$.  

It has recently been understood that the $c$-vectors play an important role
in the behaviour of a cluster algebra associated to $B^0$.  
Nakanishi and 
Zelevinsky showed in \cite{NZ} that, once it is established that the 
$c$-vectors are sign-coherent, meaning that, for each $c$-vector,
either all the entries are non-negative or all are non-positive, then 
several fundamental results on the corresponding cluster algebra
follow by an elementary argument (specifically, Conjectures 1.1--1.4 of \cite{DWZ}).  

In this paper, we give a representation-theoretic interpretation of 
the $c$-vectors as classes in the Grothendieck group 
of indecomposable objects in the bounded derived category of a hereditary
abelian category.  
Their sign-coherence is an immediate 
consequence of this description.  

We use our representation-theoretic interpretation of $c$-vectors
to give a purely combinatorial description of which collections of
vectors arise as the collection of 
$c$-vectors for some cluster associated to $B^0$: they are certain collections
of roots in the root system associated to $B^0$.  (A more precise statement
is given in Section \ref{combcharsection}.)  

We emphasize that sign-coherence of 
$c$-vectors is already known more generally than the setting in
which we work, so sign-coherence does not constitute a new result.  The novelty
here consists in our approach, which uses a relatively light 
theoretical framework, and in the characterizations of 
the sets of $c$-vectors that can appear, which are new.

\subsection{Description of the categorification} \label{CatSummary}
Starting from $B^0$, we will define a certain hereditary category $\S$ (the 
definition appears in Section \ref{defS}).  Write 
$D^b(\S)$ for the bounded derived category of $\S$. 

As those familiar with derived categories will know, the bounded derived
category of a hereditary category is very easy to work with. We review this in the appendix to this paper. If the reader is fearful of derived categories, we urge him or her to turn there now. (The first author suffered from similar fears until a year ago.)
In particular, we recall that the indecomposable objects of $D^b(\S)$ are of the form $M[i]$ where 
$M$ is an indecomposable object of $\S$, and $i\in \mathbb Z$.  
For $M,N \in \S$, we have:
$$\Ext^r_{D^b(\S)}(M[i], N[j]) \cong \Ext^{r-i+j}_{\S}(M,N).$$

An object $X$ in $D^b(\S)$ is called exceptional if it is indecomposable 
and $\Ext^1(X,X)=0$.
So such an $X$ must be of the form $M[i]$, where $M$ is indecomposable 
and $\Ext^1(M,M)=0$. (We note for the record that the $0$ object is not indecomposable.)

We write $K_0(\mathcal{S})$ for the Grothendieck group of $\S$; for $X \in \S$, we write $[X]$ for the class of $X$ in $K_0(\S)$. 
For a complex $X_{\bullet}$ in $D^b(\S)$, we write $[X_{\bullet}]$ for $\sum (-1)^i [X_i]$; this map is well defined on isomorphism classes of objects in $D^b(\S)$, and is additive on triangles in the natural way.

 We will write 
$S_1,\dots,S_n$ for the simple objects of $\mathcal S$. 
The classes $[S_1],\dots,[S_n]$ form a basis for $K_0(D^b(\S))$, 
and we shall use
this basis to identify this Grothendieck group with $\mathbb Z^n$.

We say that $(X_1,\dots,X_r)$ is an exceptional sequence in $D^b(\mathcal S)$ 
if each $X_i$ is exceptional and $\Ext^{\bullet}(X_j,X_i)=0$ for $j>i$.  
The maximum length of an exceptional sequence is $n$, the number of 
simples of $\S$; a maximal-length exceptional sequence is called 
complete.  

We call a complete exceptional sequence $(X_1,\dots,X_n)$ noncrossing if it
has the following properties:
\begin{itemize} 
\item Each $X_i$ is in either $\S$ or $\S[-1]$.
\item $\Hom(X_i,X_j)=0=\Ext^{-1}(X_i,X_j)$ for $i\ne j$.  
\end{itemize}

We prove the following:

\begin{theorem}\label{repdesc} Let $B^0$ be a skew-symmetrizable matrix.  A collection $C$ of
$n$ vectors in $\mathbb Z^n$ is the collection of $c$-vectors
for some $\vertex\in\Tree$ if and only if there is a noncrossing  
exceptional sequence $(V_1,\dots,V_n)$ in $\S$, such that 
$C$ consists of the classes in $K_0(\S)$ of the objects $X_i$.  

Moreover, we can recover the top half of the corresponding $\tilde{B}$ matrix as an alternating combination of certain $\Ext$ groups, see Theorem~\ref{GoodStuffRepThry} for details.
\end{theorem}

\begin{remark} There is at most one exceptional object of $\S \cup \S[-1]$ in a given $K_0$-class, so this exceptional sequence is unique up to reordering. \end{remark}

Let $M$ be an indecomposable object of $\S$.  Then $[M]$ is a 
non-negative linear combination of the classes $[S_1],\dots,[S_n]$, and
$[M[i]]=(-1)^i[M]$, so $[M[i]]$ is sign-coherent.  
Therefore, Theorem \ref{repdesc} implies in particular that the
$c$-vectors are sign-coherent.  This is the essential ingredient required
for the machinery developed by Nakanishi and Zelevinsky in \cite{NZ} to 
be applicable.  Given this fact, they provide a (suprisingly short and
elementary) deduction of Conjectures 1.1--1.4 of \cite{DWZ}
(reformulating conjectures of Fomin and Zelevinsky from \cite{CA4}).

\begin{corollary}
Let $v_0$ be an acyclic seed of a cluster algebra and let $v_1$ be some other seed.
Conjectures~1.1, 1.2 and~1.4 of~\cite{DWZ} hold with $t_0=v_0$ and $t=v_1$. 
Conjecture~1.3 of~\cite{DWZ} holds with $t_0=v_1$ and $t=v_0$.
\end{corollary}

These conjectures were already known to hold in this case.  For acyclic skew-symmetric
cluster algebras (among others) they were first shown by Fu and Keller \cite{FK}.  They
have subsequently been shown for arbitrary skew-symmetric cluster algebras by 
\cite{DWZ,Pla}. They were established for a subset of skew-symmetrizable cluster algebras
including the acyclic cluster algebras by \cite{dem} (extending 
techniques of \cite{DWZ}).  The conjectures were also proved by Nagao
\cite{Nag} in the skew-symmetric case under an additional technical 
assumption.  
These papers all use heavy 
machinery of some kind: \cite{FK,Pla} use general 2-Calabi-Yau triangulated
categories, \cite{DWZ,dem} use representations of quivers with potentials,
and \cite{Nag} uses Donaldson-Thomas theory.  We prove less, but 
get away with a lighter theoretical structure --- essentially just the
representation theory of hereditary algebras, mainly drawing on 
\cite{rin2}.

The idea of using Nagao's approach to understand acyclic
(skew-symmetric) cluster algebras has been carried out in \cite{KQ}.
Like the present paper, that paper also focusses on the
collections of objects which we view as noncrossing exceptional
sequences, but from a somewhat different perspective.

\subsection{The combinatorial characterization of $c$-vectors} \label{combcharsection}

In this section, we state
a necessary and sufficient combinatorial
criterion for a collection of vectors to be the $c$-vectors associated
to some $v\in\Tree$.  

There is a symmetric bilinear form $(\ ,\ )$ on $K_0(\S)$, given by
$$([A], [B]) = \sum (-1)^j \dim_{\kk} \Ext^j(A,B) + \sum (-1)^j \dim_{\kk} \Ext^j(B,A).$$
We will meet a nonsymmetric version of this form, called $E(\ ,\ )$, in Section~\ref{decategorify}.

There is a reflection group $W$ which acts naturally on the Grothendieck group
preserving this symmetrized form.  It is generated by the reflections
$s_i$ corresponding to the classes of the simple objects $[S_i]$. 
For any exceptional object $E$, the group $W$ contains the reflection 
$$t_{[E]}(v)=v-\frac{2([E],v)}{([E],[E])}[E].$$

This gives rise to a root system inside $K_0(\S)$, consisting of all 
elements of the form $w[S_i]$ for $w\in W$ and $1\leq i \leq n$.  

\begin{theorem}\label{characterize} 
A collection of $n$ vectors $v_1,\dots,v_n$ in 
$\mathbb Z^n$ is the 
set of $c$-vectors for some cluster if and only if:
\begin{enumerate}
\item[(1)] The vectors $v_i$ are roots in the root system associated to
$\S$.
\item[(2)] If $v_i,v_j$ are both positive roots or both negative roots, then 
$(v_i,v_j)\leq 0$.  
\item[(3)] It is possible to order the vectors so that the positive vectors
precede the negative vectors, and the product of the reflections corresponding
to these vectors, taken in this order, equals $s_1\dots s_n$.  
\end{enumerate}
\end{theorem}

\subsection{Compatibility of notation with the authors' other work} \label{signsec}

Both authors have written several other papers related to the present work. 
The notations in this paper are entirely compatible with the first author's notations in~\cite{RS1} and~\cite{RS3}. 
In order to achieve this, it is necessary in Section~\ref{decategorify} to define $E([X], [Y])$ to be $\sum (-1)^r \dim \Ext^r(Y,X)$, rather than the more natural seeming $\sum (-1)^r \dim \Ext^r(X,Y)$. Note that these two papers never refer to a quiver, so the choice of which oriented quiver corresponds to a given $B$-matrix is not established in those papers. 
In~\cite{RS2}, the opposite relationship between $B$-matrices and quivers is chosen.
So this paper is compatible with~\cite{RS1} and~\cite{RS3}, and~\cite{RS2} is likewise compatible with~\cite{RS1} and~\cite{RS3}, but this paper is not compatible with~\cite{RS2}.

The notations in this paper are entirely compatible with the second author's work in~\cite{mmm} and~\cite{art}, except for a minor difference noted
in Section \ref{murev}.

\section{Valued quivers and exceptional sequences}

In this section, we explain the representation-theoretic objects which we will use.

\subsection{Definition of the category $\S$}\label{defS}
Our fixed integer matrix $B^0$ is skew-symmetrizable, which means that 
$-(B^0)^TD=DB^0$ for some positive integer diagonal matrix $D$, with diagonal 
entries $d_1,\dots,d_n$.  For convenience, we assume that the enties
$(B^0)_{ij}=b^0_{ij}$ are positive when $i<j$.    


We will use this data to construct a $\kk$-linear category $\S$, for some field $\kk$.
For us, the internal structure of the objects of $\S$ is irrelevant.
What is important is that 
\begin{enumerate}
\item $\S$ is a hereditary category which has a simple object $S_i$ for each vertex of our quiver 
\item $K_i := \End(S_i)$ is a field, with $\dim_{\kk} K_i = d_i$ 
\item $\dim_{K_i} \Ext^1(S_i, S_j) = b^0_{ij}$ for $i<j$  and, thus, $\dim_{K_j}  \Ext^1(S_i, S_j) = -b^0_{ji}$
\item $\Ext^1(S_i, S_j) = 0$ for $i \geq j$.
\end{enumerate}

The category $\mathcal S$ 
will be the representations of a certain valued quiver
$Q$. We provide a whirlwind description of valued quivers; for a more in depth discussion, see~\cite[Chapter 3]{DDPW}.
When $B$ is skew-symmetric, and $d_1=\cdots=d_n=1$, this is the more familiar construction of representations of a standard quiver.

\def\LCM{\mathrm{LCM}}
Let $\kk$ be a field for which we can make the following constructions: Let $K_i$ be an extension of $\kk$ of degree $d_i$ and let $E_{ij}$ be a $K_i \otimes_\kk K_j$-bimodule which has dimension $b^0_{ij}$ over $K_i$.
One way to achieve this is to take $\kk = \mathbb{F}_{p}$ and $K_i = \mathbb{F}_{p^{d_i}}$. Then let $E_{ij}$ be a $\mathbb{F}_{p^{\LCM(d_i, d_j)}}$ vector space of dimension $d_i b^{0}_{ij}/ \mathrm{LCM}(d_i, d_j)$ and let $K_i$ and $K_j$ act on $E_{ij}$ by the embeddings of $K_i$ and $K_j$ into $\mathbb{F}_{p^{\LCM(d_i, d_j)}}$. 

A representation of $Q$ consists of a $K_i$ vector space $V_i$, associated to the vertex $i$ of our quiver, and a map $E_{ij} \to \mathrm{Hom}(V_i, V_j)$ which is both $K_i$-linear and $K_j$-linear.
The category of such representations is $\S$.

%

\subsection{Examples of noncrossing exceptional sequences} \label{Examples}

Recall the definition of noncrossing exceptional sequences from subsection~\ref{CatSummary}.

\begin{example} \label{RunningExample}
Consider the quiver $v_1 \to v_2 \to v_3 \to v_4 \to v_5$. 
We write $\alpha_i$ for the dimension vector of the simple object $S_i$.
Consider the sequence of roots $(\alpha_1, \alpha_2+\alpha_3, \alpha_4 + \alpha_5, - \alpha_2, - \alpha_4)$. 
This sequence obeys the conditions of Theorem~\ref{characterize}. 

The corresponding sequence of objects in $D^b(\S)$ is $(S_1, A_{23}, A_{45}, S_2[-1], S_4[-1])$, where $A_{i(i+1)}$ is the quiver representation which has one dimensional vector spaces in positions $i$ and $i+1$ and a nonzero map between them. 
There are six nontrivial $\Ext$ groups:
\[ \begin{array}{l l l l}
\Ext^1(S_1, A_{23}) & & \Ext^2(S_1, S_2[-1]) &\cong \Ext^1(S_1, S_2) \\
\Ext^1(A_{23}, A_{45}) & & \Ext^1(A_{23}, S_{2}[-1]) &\cong \Hom(A_{23}, S_2) \\
\Ext^2(A_{23}, S_4[-1]) &\cong \Ext^1(A_{23}, S_4)\quad & \Ext^1(A_{45}, S_4[-1]) & \cong \Hom(A_{45}, S_4) \\
\end{array} \]
It is now easy to verify that this sequence is exceptional and noncrossing.

%
\end{example}

\begin{example} \label{A2Example}
Consider the representations of the quiver $v_1 \to v_2$. There are three indecomposable representations: the simple modules $S_1$ and $S_2$, and one other which we call $A_{12}$. The noncrossing exceptional sequences are
$$(S_1, S_2),\ (A_{12}, S_1[-1]),\ (S_2, A_{12}[-1]),\ (S_1[-1], S_2[-1]),\ (S_1, S_2[-1]).$$
\end{example}

\section{Background on cluster algebras}\label{ca_background}

The first ingredient for a cluster algebra is an $n\times m$ matrix $B$, with
$m\geq n$, such that the {\it principal part}, the first $n$ rows of the 
matrix, is skew-symmetrizable.  
Second, we start with a collection $x_1,\dots,x_m$ of algebraically independent 
indeterminates in a field $\mathcal F$.  
We assign $(B,(x_1,\dots,x_m))$ to a vertex
$\vertex_b$ of an infinite $n$-regular tree $\Tree$.  If $\vertex'$ is adjacent, along an edge labelled $i$, 
to a vertex 
$\vertex$ labelled by $(B^\vertex,(x^\vertex_1,\dots,x^\vertex_m))$, the {\it mutation} rule tells us how 
to calculate $B^{\vertex'}=\mu_i(B^{\vertex})$ and $x_1^{\vertex'},\dots,x_m^{\vertex'}$.  

In this paper, we will not need to make direct reference to the cluster
variables $x_i^\vertex$ themselves, and so we will only discuss the mutation rule for
matrices. 
That rule is as follows:

$$\mu_i(B)_{jk} = \left\{ 
\begin{array}{ll} -B_{jk} & \textrm { if $j=i$ or $k=i$} \\
                  B_{jk} + [B_{ji}]_+[B_{ik}]_+ - [B_{ji}]_-[B_{ik}]_-
& \textrm{ otherwise}\end{array} \right.$$
where $[a]_+=\max(a,0)$ and $[a]_-=\min(a,0)$.  

We are, in particular, interested in the following situation.  Let $B^0$ 
be (as we have already supposed) an $n\times n$ skew-symmetrizable matrix,
and let $\widetilde B^0$ be the 
the $2n\times n$ matrix whose top half
is $B^0$, and whose bottom half is the $n\times n$ identity matrix.  We
assign this matrix to the vertex $\vertex_b$ of $\Tree$.
The mutation rule now assigns to each vertex $\vertex$ 
of $\Tree$ some $2n\times n$ matrix $B^\vertex$.  The $i$-th column of the bottom half of this
matrix is denoted $c_i^\vertex$.  

\begin{example} \label{trivial}
We illustrate these ideas by listing the matrices which are obtained
for the $2\times 2$ skew-symmetric matrix $B^0$ with $B^0_{12}=-B^0_{21}=1$.  
We start with $\widetilde{B}^0$, and proceed to mutate
alternately at the two possible positions, starting with the first.  
$$\left[\begin{array}{rr} 0&1\\
                        -1&0\\
                        \hline
                         1&0\\
                         0&1\end{array}\right] \to 
\left[\begin{array}{rr} 0&-1\\
                        1&0\\
                        \hline
                        -1&1\\
                         0&1\end{array}\right]  \to 
\left[\begin{array}{rr} 0&1\\
                        -1&0\\
                        \hline
                        0&-1\\
                        1&-1\end{array}\right] \to 
\left[\begin{array}{rr} 0&1\\
                        -1&0\\
                        \hline
                        0&-1\\
                        -1&0\end{array}\right]  \to 
\left[\begin{array}{rr} 0&-1\\
                        1&0\\
                        \hline
                        0&1\\
                        -1&0\end{array}\right]$$
                        Note that the columns of the bottoms of these matrices correspond to the dimension vectors of the terms in the noncrossing sequences from Example~\ref{A2Example}, but that the exceptional ordering of that example is not always the order of the columns of the matrix.   
\end{example}

\section{Background on exceptional sequences}\label{recall}



\subsection{The mutation operators} \label{mutation}

Recall the following lemma:
\begin{lemma}[{\cite[Lemma 1.2]{art}}]\label{onlyone} If $(E,F)$ is an exceptional sequence, there is at most one
$j$ such that $\Ext^j(E,F)\ne 0$.  
\end{lemma}

There are well-known mutation operations on exceptional sequences which we will now recall.  
The operator $\mu_i$ acts on an exceptional sequence whose $i$-th and $i+1$-st terms
are $X_i$ and $X_{i+1}$ by replacing the subsequence $(X_i,X_{i+1})$ by
$(X_{i+1},Y)$, where $Y$ is determined by $X_i$ and $X_{i+1}$.
We also describe $\mu_i$ as ``braiding $X_{i+1}$ in front of $X_i$''.  
This is intended to suggest a braid diagrams: in a braid diagram, the front 
string
is drawn unbroken, while the back string is drawn broken just as $X_{i+1}$ moves past $X_i$ and changes $X_i$, while remaining unchanged itself.
There are also inverses of the $\mu_i$: the operation $\mu_i^{-1}$ braids 
$X_i$ in front of $X_{i+1}$.

\begin{remark} Only certain exceptional sequences have $B$-matrices associated to them and mutating such an exceptional sequence does not always produce another such.
So the use of these mutation operators does not always correspond to a mutation of $B$-matrices.
The terminology ``mutation" is very standard in both cases.
\end{remark}

We now define $\mu_i$ and $\mu_i^{-1}$ precisely.  

If $\Ext^j(X_i, X_{i+1})=0$ for all $j$, then $\mu_i$ switches $X_i$ and $X_{i+1}$.

Suppose now that $k$ is the unique index such that $\Ext^k(X_i, X_{i+1}) \neq 0$, so $\Hom(X_i, X_{i+1}[k]) \neq 0$. 
Let $H:= \Hom(X_i, X_{i+1}[k])$ and let $H^{\vee}$ be the dual vector space. Then we have a universal map $X_i \to X_{i+1}[k] \otimes H^{\vee}$, called the ``left thick $X_{i+1}$ approximation to $X_i$.''
We complete this to a triangle
$$Y \rightarrow X_i \rightarrow X_{i+1}[k] \otimes_{\End(X_{i+1})} H^{\vee} \rightarrow.$$
Then $\mu_i$ replaces the subsequence $(X_i,X_{i+1})$ by $(X_{i+1},Y)$.
Similarly, we have a right thick approximation $X_i[-k] \otimes H \to X_{i+1}$.
Complete this to a triangle
$$X_i[-k] \otimes_{\End(X_{i})} H \to X_{i+1} \to Z \to.$$
The operation $\mu_i^{-1}$ replaces $(X_i, X_{i+1})$ by $(Z, X_i)$. 

The operations $\mu_i$ and $\mu_i^{-1}$ are morally inverse.
More precisely, $\mu_i \mu_i^{-1}$ replaces $(X_i, X_{i+1})$ by $(X'_i, X_{i+1})$ where $X'_i$ is isomorphic to $X_i$ in $D^b(\S)$. Similarly, $\mu_i^{-1} \mu_i$ replaces $(X_i, X_{i+1})$ by $(X_i, X'_{i+1})$ where $X'_{i+1}$ is likewise isomorphic to $X_{i+1}$.
The operation of completing to a triangle is only defined up to isomorphism, so this is the best statement we can make; readers with experience in triangulated categories will be familiar with the subtleties here. 
A paper which treats this carefully is~\cite{GK}; for our present purposes, we can ignore this issue and treat $\mu_i$ and $\mu_i^{-1}$ as inverse.

\begin{lemma} \label{MutationReflection}
We have $[Y]=t_{[X_{i+1}]} [X_i]$ and $[Z] = t_{[X_i]} [X_{i+1}]$, where $t_{[E]}$ is the reflection defined in Section~\ref{combcharsection}. 
\end{lemma}

\begin{proof}
We make the computation for $Y$; the case of $Z$ is similar.
We abbreviate $\End(X_{i+1})$ to $L$. 

Using additivity of dimension vectors in a triangle, we have
\[ \begin{array}{r l}
[Y] =& [X_i] - (\dim_L H^{\vee} ) \cdot [X_{i+1}[k]]  \\
 =& [X_i] - (\dim_L \Hom(X_i, X_{i+1}[k])) \cdot  (-1)^k [X_{i+1}] \\
 =& [X_i] - (-1)^k (\dim_L \Ext^k(X_i, X_{i+1})) \cdot [X_{i+1}] \\
 =&  [X_i] - (-1)^k \frac{\dim_{\kk} \Ext^k(X_i, X_{i+1})}{\dim_{\kk} L}  \cdot [X_{i+1}]
 \end{array} \]
Using Lemma~\ref{onlyone} and the definition of an exceptional sequence, we have
$$(-1)^k \dim_{\kk} \Ext^k(X_i, X_{i+1}) = ([X_i], [X_{i+1}]).$$
Since $X_{i+1}$ is exceptional, we have $\Ext^r(X_{i+1}, X_{i+1}) = 0$ for $r \neq 0$ and thus $([X_{i+1}], [X_{i+1}]) = 2 \dim_{\kk} L$.

So $[Y] = [X_i] - \frac{2 ([X_i], [X_{i+1}]) }{([X_{i+1}], [X_{i+1}]) }\cdot  [X_{i+1}] = t_{[X_{i+1}]} [X_i]$ as desired.
\qed\end{proof}

The mutation operations satisfy the braid relations, meaning that
$\mu_i\mu_j=\mu_j\mu_i$ if $|i-j|>1$ and 
$\mu_i\mu_{i+1}\mu_i=\mu_{i+1}\mu_i\mu_{i+1}$.





\subsection{Some needed results}

Let $\mathcal E =(E_1,\dots,E_n)$ be an
exceptional sequence in $D^b(\S)$.  Let $\overline E_i$ be the shift of 
$E_i$ which lies in $\S$.  Note that $(\overline E_1,\dots,\overline E_n)$
is still an exceptional sequence.  
Define the Hom-Ext quiver of $\mathcal E$
to be the quiver on vertex set $1,\dots,n$, where there is an arrow from
$i$ to $j$ if $\Hom(\overline E_i,\overline E_j)\ne 0$ or $\Ext^1(\overline E_j,\overline E_i)\ne 0$.  (Note that
the orders of the terms in the $\Hom$ and the $\Ext^1$ are different!)

\begin{proposition}[{\cite[Theorem 1.4]{mmm}}]\label{acyclicity} The Hom-Ext quiver of an exceptional sequence is 
acyclic.  \end{proposition}

In other words, the Hom-Ext quiver can be understood as defining a poset.  The 
intuition for this poset is that $i$ precedes $j$ iff $\overline E_i$ ``comes earlier in the 
AR quiver'' than $\overline E_j$.  This expression is in scare quotes because in affine type, the
AR quiver has cycles, and in wild type, there are many morphisms which 
are not recorded in the AR quiver.


The following results are standard:

\begin{lemma}\label{isom} If $(B,C)$ is an exceptional sequence, and 
$\mu_1(B,C)=(C,B')$, then 
$\End(B)\isom \End(B')$.  
\end{lemma}

\begin{lemma}\label{equiv} If $(A,B,C)$ is an exceptional sequence, and
$(C,A',B')=\mu_1\mu_2(A,B,C)$, then $\Ext^j(A',B')\isom\Ext^j(A,B)$ for
all $j$.
\end{lemma}

\begin{lemma} \label{vanishok}
Let $(A,B,C,D)$ is an exceptional sequence, and let $(A, C', B, D) = \mu_2 (A,B,C,D)$. If $\Ext^r(A, B)=\Ext^r(A,C)=0$ then $\Ext^r(A, C')=0$. Similarly, if $\Ext^r(B,D)=\Ext^r(C,D)=0$ then $\Ext^r(C', D)=0$.
\end{lemma}

\begin{proof}[of Lemmas~\ref{isom} and~\ref{equiv}] Starting with 
an exceptional sequence $(A,B,C)$ and braiding $C$ in front to $(C,A',B')$
defines an equivalence of categories
from the triangulated, extension-closed subcategory generated by $A$ and $B$ to
to that generated by $A'$ and $B'$. Both results are isomorphisms between a Hom group in one of these categories to a Hom group in the other.
\qed\end{proof}

\begin{proof}[of Lemma~\ref{vanishok}]
As there is a triangle $B \to C' \to C$, this follows from the long exact sequence of $\Ext$ groups.
\qed\end{proof}


\subsection{The cluster complex and $\murev$} \label{murev}

Define $$\murev=[1](\mu_{n-1})(\mu_{n-2}\mu_{n-1})\dots(\mu_{2}\dots \mu_{n-1})(\mu_1\dots\mu_{n-1})$$
(i.e., first apply the sequence of mutations to the exceptional sequence,
and then apply $[1]$ to all the terms in the sequence.)  Note that this
differs by $[1]$ from the definition in \cite{art}.

\begin{lemma}\label{comrev} $\murev\mu_i=\mu_{n-i}\murev$ \end{lemma}
\begin{proof} This is a standard calculation in the braid group.
\qed\end{proof}

We say that a complete 
exceptional sequence is a \emph{cluster exceptional sequence}
if 
its terms lie in $\mathcal S\cup \{\text{ the projective indecomposable
objects of }\S[1]\}$, and $\Ext^1(A,B)=0$ for any $A,B$ in the sequence.  
We write $P_i$ for the indecomposable projective generated at vertex $i$.

We now define the \emph{cluster complex}.
This is the simplicial complex whose vertices are isomorphism classes of exceptional indecomposable objects in $\S$, together with a vertex for each of the projective indecomposable objects of $\S[1]$. A collection of such objects forms a face of the cluster complex if they can appear together in a cluster exceptional sequence.

Hubery \cite{hub} studies the same complex under a slightly different 
definition.  Since we need some of his results, we now describe his approach and its relation to ours.
Define the {\it completed tilting complex} to be a simplicial complex on
the exceptional indecomposables of $\S$, together with the 
positive integers $1\dots n$.  In the completed tilting complex,
$T_1,\dots,T_{j},i_1,\dots,i_r$ forms a 
face if $\Ext^1(\bigoplus_k T_k,\bigoplus_k T_k)=0$ and, for all $k$, 
we have that $T_k$ is not 
supported over any of the vertices $i_1,\dots,i_r$.  

\begin{lemma} The cluster complex and the completed tilting complex are
isomorphic, under the map taking the indecomposable $E$ to itself, and
taking $P_i[1]$ to $i$.  \end{lemma}

\begin{proof} Consider a face of the cluster complex, say 
$T_1,\dots,T_s,P_{i_1}[1],\dots,P_{i_r}[1]$, with $T_k\in\S$ for all $k$.  
By definition,
$\Ext^1(T_i,T_j)=0$.  
Then
$$0=\Ext^1(P_{i_j}[1],T_i)=\Hom(P_{i_j},T_i),$$
so $T_i$ is not supported over vertex $i_j$.  
This shows that there is an inclusion from the cluster complex to the
completed tilting complex.  

Conversely, consider a face $T_1,\dots,T_s,i_1,\dots,i_r$ of the 
completed tilting complex.  Define $\widetilde Q$ to be the quiver $Q$
with the vertices $i_1,\dots,i_r$ removed, and define $\widetilde {\S}$ 
similarly.  Then
$\bigoplus T_i$ is a partial tilting object for $\widetilde{\S}$, so it is
a direct summand of a tilting object $\overline T$ 
for $\widetilde{\S}$.  Since the
Gabriel quiver of a tilting object has no cycles, the direct summands of 
$\overline T$ can be ordered into an exceptional sequence.  Appending
$P_{i_1}[1],\dots,P_{i_r}[1]$ onto the end, we obtain a cluster exceptional
sequence.  
\qed\end{proof}

We now recall the main results of \cite{hub} and \cite{art}, appropriately specialized.

\begin{theorem}[{\cite[Theorem 19]{hub}}]\label{connected}
\begin{enumerate}\item 
Any $(n-1)$-dimensional face of the completed tilting complex is 
contained in exactly two $n$-dimensional faces. 
\item It is possible to pass from any $n$-dimensional face of the 
completed tilting complex to any other $n$-dimensional face by a 
sequence of steps moving from one $n$-dimensional face to an
$n$-dimensional face adjacent across an $(n-1)$-dimensional face.  
\end{enumerate}
\end{theorem}



\begin{theorem}[{\cite[Theorem 6.9]{art}}]\label{bijection} 
The map $\murev$ is a bijection from noncrossing exceptional sequences to
cluster exceptional sequences.   
\end{theorem} 

If $X_i$ and $X_{i+1}$ are two consecutive 
terms of an exceptional sequence such that
$\Ext_\bullet(X_i,X_{i+1})=0$, then interchanging $X_{i+1}$ and $X_i$ clearly gives another exceptional sequence. 
We will call such a trivial reordering a \emph{commutation move} and say that  two exceptional sequences are \emph{commutation equivalent} if they can be obtained from each other by a sequence of commutation moves.
Observe that, if a set of exceptional objects in $\S$ has two exceptional orderings, then the two orderings must be commutation equivalent.

The following lemma follows from the proof of  \cite[Theorem 5.2]{art}:
\begin{lemma}[\cite{art}]  \label{RevEquiv}
The maps $\murev$ and $\murev^{-1}$ take commutation equivalent sequences to commutation equivalent sequences
\end{lemma}

We now explain the effect of combining the results of~\cite{hub} and~\cite{art}.

Let $(X_1, \ldots, X_n)$ and $(Y_1, \ldots, Y_n)$ be two complete exceptional sequences.
We say that $Y_{\bullet}$ is obtained by \emph{noncrossing mutation} of $X_{\bullet}$ at $X_i$ if
\begin{enumerate}
\item[(1)] $X_{\bullet}$ and $Y_{\bullet}$ are not commutation equivalent.
\item[(2)] $X_{\bullet}$  and $Y_{\bullet}$ are noncrossing.

Either
\item[(3a)] $X_i \in \S$ and $Y_{\bullet}$ is obtained from $X_{\bullet}$ by possibly applying some commutation moves, braiding $X_i$ over $(X_{i+1}, \ldots, X_j)$ for some index $j$, replacing $X_i$ by $X_i[-1]$ and possibly applying some commutation moves again or
\item[(3b)] $X_i \in \S[-1]$ and $Y_{\bullet}$ is obtained from $X_{\bullet}$ by possibly applying some commutation moves, braiding $X_i$ over $(X_{j}, \ldots, X_{i-1})$ for some index $j$, replacing $X_i$ by $X_i[1]$ and possibly applying some commutation moves again.
\end{enumerate}

\begin{proposition} \label{OneMutate}
Given a  noncrossing exceptional sequence $X_{\bullet}$ and an element $X_i$ in it, there is at most one commutation class of exceptional sequences which can be obtained from $X_{\bullet}$ by noncrossing mutation at $X_i$.
\end{proposition}

In Lemma~\ref{CanMutate}, we will show there is exactly one such sequence. 

\begin{proof}
Suppose that $Y_{\bullet}$ and $Z_{\bullet}$ could both be so obtained. Let $X'_{i}$ be the element of $\murev(X_{\bullet})$ corresponding to $X_i$.
Then $\murev(Y_{\bullet})$ and $\murev(Z_{\bullet})$ are both obtained from $\murev(X_{\bullet})$ by braiding $X'_i$ under some subset of $\murev(X_{\bullet})$ .
In particular, $\murev(Y_{\bullet})$ and $\murev(Z_{\bullet})$ both contain all the elements of $\murev(X_{\bullet})$ other than $X'_i$. 
By Theorem~\ref{bijection},  $\murev(X_{\bullet})$, $\murev(Y_{\bullet})$ and $\murev(Z_{\bullet})$ are all clusters so, by Theorem~\ref{connected}, the underlying sets of  $\murev(Y_{\bullet})$ and $\murev(Z_{\bullet})$ are the same. So  $\murev(Y_{\bullet})$ and $\murev(Z_{\bullet})$ are commutation equivalent and, by Lemma~\ref{RevEquiv}, so are $Y_{\bullet}$ and $Z_{\bullet}$.
\qed\end{proof}

\begin{remark}
We explain why it is not obvious from Theorem~\ref{connected} that such a sequence exists.
Let $n=3$, let $(X,Y,Z)$ be a noncrossing exceptional sequence and suppose that we want to perform a noncrossing mutation at $Y$. 
Let $(A,B,C) = \murev(X,Y,Z)$ and let $\{ A, C, D \}$ be the elements of the other cluster containing $A$ and $C$.
One would hope that this other cluster is obtained by braiding $B$ behind one of $A$ and $C$, in which case applying $\murev^{-1}$ would give a noncrossing exceptional sequence which differs from $(X,Y,Z)$ by braiding $Y$ over one of $X$ and $Z$. If so, then this is a noncrossing mutation at $Y$, as desired.

However, suppose now that there are no $\Ext$'s between $A$ and $C$. It is \emph{a priori} possible that the exceptional ordering of $\{ A, C, D \}$ is $(C,D,A)$, obtained by braiding $B$ behind $C$ to obtain $(A,C,E)$, commuting $C$ and $A$, and then braiding $E$ behind $A$.
In this case, the hope of the previous paragraph fails.
The essence of the proof of Lemma~\ref{CanMutate} is ruling this case out.
\end{remark}


\begin{example} \label{RunningExample2}
In Example \ref{RunningExample}, we gave an example of a noncrossing 
exceptional sequence, $(S_1, A_{23}, A_{45}, S_2[-1],S_4[-1])$ with corresponding roots $(\alpha_1, \alpha_2+\alpha_3, \alpha_4+\alpha_5, - \alpha_2, - \alpha_4)$.  

We will braid $A_{23}$ over $A_{45}$ and $S_2[-1]$ and replace $A_{23}$ by
$A_{23}[-1]$.  
This results in the new sequence $(S_1, A_{2345}, S_3, A_{23}[-1], S_4[-1])$ with corresponding sequence of roots $(\alpha_1, \alpha_2+\alpha_3 + \alpha_4 + \alpha_5, \alpha_3, - \alpha_2 - \alpha_3, - \alpha_4)$. 
The somewhat ambitious reader may verify that this new sequence again obeys the conditions of Theorem~\ref{characterize}; the more ambitious reader can check that the corresponding sequence of objects truly is again exceptional.  It 
follows that in going between these two sequences we have effected a 
noncrossing mutation.  

Note that the sequences $(S_1, A_{23}, A_{45}, S_2[-1], S_4[-1])$  and   $(S_1, A_{23}, S_2[-1], A_{45}, S_4[-1])$ differ by a commutation move, so $(S_1, A_{23}, S_2[-1], A_{45}, S_4[-1])$  can also be turned into the sequence $(S_1, A_{2345}, S_3, A_{23}[-1], S_4[-1])$ by noncrossing mutation at $A_{23}$.
\end{example} 

\begin{example} \label{RunningExample3}
We consider the two noncrossing exceptional sequences
from Example~\ref{RunningExample2},
$(S_1, A_{23}, A_{45}, S_2[-1],S_4[-1])$, and $(S_1, A_{2345}, S_3, A_{23}[-1], S_4[-1])$.  We now consider $\murev$ applied to these sequences, and show
that they are cluster exceptional sequences related by a cluster mutation.  

$$\murev(S_1, A_{23}, A_{45}, S_2[-1],S_4[-1])=
(S_4,S_2,S_5[1],A_{345}[1],A_{12345}[1])$$
Since $A_{i(i+1)\dots 5}=P_i$, the objects of this sequence do lie in
$\S\cup \{P_i[1]\}$; we leave it as an exercise for the reader to check
that this is a cluster exceptional sequence.  

$$\murev(S_1, A_{2345}, S_3, A_{23}[-1], S_4[-1])=
(S_4,A_{234},S_2,S_5[1],A_{12345}[1]).$$
Again, it is easy to check that the objects lie in the appropriate set,
and it is clear that the two sequences differ in one object.
\end{example}

\begin{example} \label{A2Example2}
Applying $\murev$ to the noncrossing exceptional sequences from Example~\ref{A2Example} gives the following cluster exceptional sequences:
$$(S_2[1], A_{12}[1]),\ (S_1, S_2[1]),\ (A_{12}, S_1),\ (S_2, A_{12}),\ (S_2, A_{12}[1]).$$
The cluster complex is a pentagon, whose edges are indexed by the above sequences.
\end{example}

We conclude the section by proving Lemma~\ref{CanMutate}:
\begin{lemma} \label{CanMutate}
Let $V_{\bullet}$ be a noncrossing sequence.
For any index $i$, it is possible to perform a noncrossing mutation at $V_i$.
\end{lemma}

\begin{proof}
We describe the case where $V_i \in \S$; the case where $V_i \in \S[-1]$ is similar.
Define 
\begin{eqnarray*}
J_1&=&\{ V_k \mid \Ext^1(V_k,V_i) \ne 0 \}\\
J_2&=&\{ V_k \mid \Ext^1(V_i,V_k)\ne 0 \} \\
J_3&=& \{ V_k \mid \Ext^2(V_i,V_k)\ne 0\}
\end{eqnarray*}

We claim that we can apply commutation moves to $V_\bullet$ so that, afterwards,
the elements of $J_1$ precede $V_i$, which precedes the elements 
of $J_2$, which, in turn, precede the elements of $J_3$. 
 
By the definition of a noncrossing sequence and the fact that $\S$ is hereditary these are the three, mutually exclusive 
possibilities for $k\ne i$ such that there is a non-zero $\Ext$ group between $V_k$ and $V_i$. 
Using the hereditary nature of $\S$, the elements of $J_1$ lie in $\S$ and the elements of $J_3$ lie in $\S[-1]$.
Let $J_2^{+}$ be $J_2 \cap \S$ and let $J_2^{-} = J_2 \cap \S[-1]$.
There can be no $\Ext$'s from elements of $\S[-1]$ to elements of $\S$, so we may apply commutations to order $J_1 \cup V_i \cup J_2^{+}$ before $J_2^{-} \cup J_3$.
Also, by definition, there are nonzero $\Ext$'s from $J_1$ to $V_i$ to $J_2^{+}$, so the fact that these elements are in the desired order automatically follows
from the fact that we have an exceptional sequence.  

Finally, we must show that we can order $J_2^{-}$ before $J_3$.
Recall the $\Hom$-$\Ext$ quiver from Proposition~\ref{acyclicity}, for
the exceptional sequence consisting of $V_i$ followed by the elements of 
$V_\bullet$ which lie in $\S[-1]$.  
Tracing through the definitions, there are arrows $J_3 \to V_i \to J_2^{-}$.
So, by Proposition~\ref{acyclicity}, there cannot be a sequence $A_0[-1], \dots,
A_m[-1]$ of objects
from $V_\bullet \cap \S[-1]$ such that $\Ext^1(A_t,A_{t+1})\ne 0$ for all $0\leq t \leq m-1$,
and such that $A_0\in J_3$, $A_1\in J_2^-$. 
It follows that we can order $J_2^{-}$ before $J_3$.

Braid $V_i$ over $J_2$, and then replace $V_i$ by $V_i[-1]$. 
Call the resulting exceptional sequence $V'_\bullet$.
We claim $V'_\bullet$ is noncrossing.
First, we check that all the $V'_j$ are in $\S \cup \S[-1]$. 
For $V_j$ not in $J_2$, this is obvious; let $V_j\in J_2$.  The approximation sequence looks like

$$V^{\oplus p}_i[-1] \rightarrow V_j \rightarrow V_j' \rightarrow V^{\oplus p}_i$$
for some $p>0$.  We see that
$V_j'$ admits a morphism from $V_j$, and a morphism to $V_i$, 
so it still lies
in $\S \cup \S[-1]$.

We now must check that $\Ext^r(V'_j, V'_k)$ vanishes for $r =0$ and $-1$. 
When $V_j$ and $V_k$ are both in $J_1 \cup J_3$, this is obvious. When they are both in $J_2$, this is Lemma~\ref{equiv}.
When one is in $J_1 \cup J_3$ and the other is in $J_2$, this is Lemma~\ref{vanishok}.
If $j=i$ and $V_k$ is in $J_1 \cup J_3$, this is obvious, and similarly with
the roles of $j$ and $k$ interchanged; when $j=i$ and $V_k \in J_2$ this follows from the definition of an exceptional sequence.
Finally, we are left with the case $k=i$ and $V_j \in J_2$. In this case, the
approximation sequence above shows that 
$\Ext^1(V'_j,V'_i)$ is nonzero, so all other $\Ext$ groups must be zero 
by Lemma~\ref{onlyone}.
\qed\end{proof}

\begin{example}
In the situation of Example~\ref{RunningExample3}, with $V_i=A_{23}$, we have $J_1 = \{ S_1 \}$, $J_2 = \{ A_{45}, S_2[-1] \}$ and $J_3 = \{ S_4[-1] \}$.
\end{example}

\section{Introduction to Frameworks}\label{decategorify}

We now describe work of Nathan Reading and the first author, regarding when the structure of a cluster algebra can be described by some Coxeter theoretic data.
Our starting point is the skew-symmetrizable matrix $B_0$, and the vector $(d_j)$.
We index the rows and columns of $B_0$ by a finite set $I$.
We now introduce the standard Coxeter theoretic terminology.

Let $V$ be a real vector space with a basis $\alpha_i$, for $i \in I$.
Let $\alpha_i^{\vee}$ be $d_i^{-1} \alpha_i$.
Define an inner product $E$ on $V$ by 
$$E(\alpha_i^{\vee}, \alpha_j) = 
\begin{cases} 
1 & \mbox{if}\ i=j \\
0 & \mbox{if}\ b_{ij} >0 \\
b_{ij} & \mbox{if}\ b_{ij}<0
\end{cases}$$

Define a symmetric bilinear form by $(\beta, \beta') = E(\beta, \beta') + E(\beta', \beta)$. 
In the theory of Coxeter groups, the form $(\cdot,\cdot)$ is the prime actor, but $E$ will return eventually.
We also define the skew-symmetric form $\omega(\beta, \beta') = E(\beta, \beta' ) - E(\beta', \beta)$.

Write $s_i$ for the reflection $\beta \mapsto \beta - 2 \alpha_i(\alpha_i, \beta)/(\alpha_i, \alpha_i)$ in $GL(V)$.
The Coxeter group $W$ is the subgroup of $GL(V)$ generated by the $s_i$. 
An element of $W$ is called a reflection if it is conjugate to one (or more) of the $s_i$.
Note that, if $b_{ij}=0$, then $\alpha_i$ and $\alpha_j$ are orthogonal
with respect to the symmetric form, so $s_i$ and $s_j$ commute.

A vector in $V$ is called a real root if it is of the form $w \alpha_i$ for some $w \in W$ and $i \in I$. 
If $\beta$ is a real root, then so is $- \beta$. 
The set of real roots is denoted $\Phi$.
A real root is called positive if it is in the positive span of the $\alpha_i$, and is called negative if it is the negation of a positive root. It is a nontrivial theorem that every real root is either positive or negative.
We'll write $\sgn: \Phi  \to \{ 1, -1 \}$ for the map which takes a root to its sign.

There is a bijection between reflections in $W$ and pairs $\{ \beta, -\beta \}$ of real roots.
Namely, if $t$ is a reflection, then its $(-1)$-eigenspace is of the form $\RR \beta$ for some real root $\beta$ and, conversely, for any real root $\beta$, the map $\gamma \mapsto \gamma - 2 \beta(\gamma, \beta)/(\beta, \beta)$ is a reflection in $W$.
We will say that $t$ is the reflection in $\beta$, or in $- \beta$.



The following definitions are from~\cite{RS3}.
A complete reflection framework consists of (1) a connected $n$-regular graph $G$ and (2) a function $C$ which, to every pair $(v,e)$ where $v$ is a vertex of $G$ and $e$ is an edge of $G$, assigns a vector $C(v,e)$ in $V$. 
One of the consequences of the axioms of a reflection framework will be that $C(v,e)$ is always a real root.
We write $C(v)$ for the $n$-tuple $\{ C(v,e) \}_{e \ni v}$.

\noindent
\textbf{Base condition:}  
For some vertex $v_b$,  the set $C(v_b)$ is the simple roots, $\{ \alpha_i \}_{i \in I}$. 

\noindent
\textbf{Reflection condition:}
Suppose $v$ and $v'$ are distinct vertices incident to the same edge $e$.
Let $C(v,e) = \beta$. Then $C(v',e) = - \beta$. Furthermore, if $t$ is the reflection in $\beta$, and $\gamma \neq \beta$ is an element of $C(v)$, then
$C(v')$ contains the root 
\[\gamma'=\left\lbrace\begin{array}{ll}
t\gamma&\mbox{if }\omega(\beta,\gamma)\ge 0,\mbox{ or}\\
\gamma&\mbox{if }\omega(\beta,\gamma)<0.
\end{array}\right.\]

For a $v$ vertex of $G$,
define $C_+(v)$ to be the set of positive roots in $C(v)$ and define $C_-(v)$ to be the set of negative roots in $C(v)$.
Let $\Gamma(v)$ be the directed graph whose vertex set is $C(v)$, with an edge $\beta \to \beta'$ if $E(\beta, \beta') \neq 0$. 

\noindent
\textbf{Euler conditions:}
Suppose $v$ is a vertex of $G$ with $\beta$ and $\gamma$ in $C(v)$.
Then
\begin{enumerate}
\item[(E1) ] If $\beta\in C_+(v)$ and $\gamma\in C_-(v)$ then $E(\beta,\gamma)=0$. 
\item[(E2) ] If $\sgn(\beta)=\sgn(\gamma)$ then $E(\beta,\gamma)\le0$.  
\item[(E3) ] The graph $\Gamma(v)$ is acyclic.\\
\end{enumerate}

\begin{remark} In any reflection framework, let $t$, $\beta$ and $\gamma$ be as in the Reflection Condition, and suppose that $\omega(\beta, \gamma) = 0$. 
By condition~(E3), either $E(\beta, \gamma)$ or $E(\gamma, \beta)$ is $0$, and $\omega(\beta, \gamma) = E(\beta, \gamma) - E(\gamma, \beta)$, so we see that $E(\beta, \gamma) = E(\gamma, \beta) = 0$. 
But then $(\beta, \gamma)=0$, so $t \gamma = \gamma$. 
We thus see that it is unimportant which of the two cases in the reflection condition is assigned the strict inequality.
\end{remark}

Given a connected $n$-regular graph $G$, and a choice of which vertex to call $v_b$, there is at most one way to put a framework on $G$;
the Reflection Condition recursively determines what $C(v,e)$ must be for every $(v,e)$. 
One then must check whether the resulting recursion is consistent, and whether or not the Euler Conditions are obeyed.

In~\cite{RS3}, it is shown that, if there is a complete reflection framework for a given initial $B$-matrix, then one can recover all the $B$-matrices and $g$-vectors of the corresponding cluster algebra from simple combinatorial operations on the framework, and many standard conjectures about cluster algebras follow in that case.
Conversely, it is also shown that, assuming certain standard conjectures about cluster algebras, every acyclic cluster algebra does come from a framework.


\section{Dimension vectors of noncrossing sequences give a framework}

We now identify the vector space $V$ (above) with $K_0(\S) \otimes \mathbb{R}$, identifying $[S_i]$ with $\alpha_i$.
We see that 
$$([X], [Y]) = E([X], [Y]) + E([Y],[X]),$$
by checking this identity on the basis of simples.
The reader may be surprised to learn that
$$E([X], [Y]) = \sum (-1)^r \Ext^r(Y, X).$$
This reversing of the order of $X$ and $Y$ is required in order to match the various sign conventions of the authors' earlier work; see section~\ref{signsec}.

Let $G$ be the graph whose vertices are commutation equivalence classes of noncrossing exceptional sequences.
Let  there be an edge from $\vertex$ to $\vertex'$ if $\vertex$ and $\vertex'$ are linked by a noncrossing mutation.  
Let $\vertex$ and $\vertex'$, joined by an edge $e$, be linked by mutating at $M \in \vertex$. 
We set $C(\vertex, e)$ be the vector $[M]$ in $V$. 
By Proposition~\ref{OneMutate}, the $C(\vertex,e)$ are distinct.
By Lemma~\ref{CanMutate}, the graph $G$ is $n$-regular.

\begin{lemma} \label{EulerLem}
The pair $(G,C)$ obey the Euler conditions.
\end{lemma}

\begin{proof}
Condition~(E3) follows from the definition of an exceptional sequence.

Note that, if $\beta = [M]$ is in $C_{+}(v)$ and $\gamma = [N]$  is in $C_{-}(v)$, then $M$ is an object of $\S$, and $N$ is an object of $\S[-1]$.
Because $\Ext^{-1}$ vanishes in a noncrossing sequence, we may apply commutation moves so that all the elements of $\S$ come before all the elements of $\S[-1]$. Then $\Ext^r(N,M)=0$ for all $r$, by the definition of an exceptional sequence, so Condition~(E1) follows.

Finally, suppose that $\beta$ and $\gamma$ are both in $C_+(v)$ (the case of $C_-(v)$ is similar). 
Then $\beta = [M]$ and $\gamma = [N]$, for two objects $M$ and $N$ in $\S$.
By the definition of a noncrossing sequence, $\Hom(N, M)=0$. Also, as $\S$ is hereditary, we have $\Ext^r(N,M)=0$ for $r \geq 2$.
So the only nonvanishing $\Ext$ group is $\Ext^1$, and we see that $E(\beta, \gamma) \leq 0$, as required by Condition~(E2).
\qed\end{proof}

\begin{corollary}
$G$ is connected.
\end{corollary}

\begin{proof}
From the proof of Proposition~\ref{OneMutate}, $G$ is a subgraph of the dual graph to the cluster complex. Since we now know that $G$ is $n$-regular, we see that $G$ is the dual graph of the cluster complex and we are done by Theorem~\ref{connected}.
\qed\end{proof}

\begin{lemma} \label{RefLem}
$(G,C)$ obeys the reflection condition.
\end{lemma}

\begin{proof}
Let vertices $\vertex$ and $\vertex'$ correspond to noncrossing sequences $V_{\bullet}$ and $V'_{\bullet}$, linked by mutation at $V_i$.
Let $[V_i] = \beta$ and let $\gamma = [V_j]$ be another element of $C(v)$.
We continue the notations $J_1$, $J_2$ and $J_3$ from the proof of Lemma~\ref{CanMutate}.

If $V_j$ is in $J_1$ or $J_3$, then $\omega(\beta, \gamma)<0$. Also, in this case, $V_j \in V'_{\bullet}$, so $\gamma \in C(v')$ as desired.
If $V_j$ is in $J_2$, then $\omega(\beta, \gamma) > 0$. Also, in this case, $V_i$ is braided over $V_j$ to obtain an element of $V'_{\bullet}$. 
So, by Lemma~\ref{MutationReflection}, $t_{\beta} \gamma \in C(\vertex')$ as desired.
Finally, we consider the case that there are no $\Ext$'s between $V_i$ and $V_j$, in which case $\omega(\beta, \gamma) =0$.
In this case, whether or not $V_i$ is braided over $V_j$, the object $V_j$ occurs in $V'_{\bullet}$, so $\gamma \in C(\vertex')$, as desired.
\qed\end{proof}

We have now checked that $G$ is connected and $n$-regular, and that the Reflection and Euler conditions hold. The Base condition is obvious, corresponding to the noncrossing partition $(S_1, \ldots, S_n)$. We conclude:
\begin{theorem} \label{NonCrossFramework}
$(G,C)$ is a complete reflection framework.
\end{theorem}

In particular, we now know
\begin{theorem}
Every acyclic cluster algebra comes from a complete reflection framework.
\end{theorem}

\section{Consequences of the Framework result} \label{FrameCons}

Recall that, in the introduction, we labeled every vertex $\vertex$ of $\Tree$ by an extended $B$-matrix $\tilde{B}^v$, related to each other by matrix mutation.

\begin{theorem} \label{GoodStuffRoots}
There is a covering map $\pi: \Tree \to G$ such that, if $\vertex \in \Tree$ and $\pi(\vertex)$ corresponds to the noncrossing sequence $(V_1, \ldots, V_n)$, then
\begin{enumerate}
\item[(1)] The columns of the bottom half of $\tilde{B}^{\vertex}$, also known as the $c$-vectors, are the $\beta_i$ in $C(\pi(\vertex))$.
\item[(2)] Reordering the rows and columns of $\tilde{B}$ to match the order of the $V_i$, we have $b^{\vertex}_{ij} = d_i^{-1} \omega(\beta_i^{\vee}, \beta_j)$.
\end{enumerate}
\end{theorem}

This is part of the main result of~\cite{RS3}. 

We can unfold the definitions of $\beta_i$ and $\omega$ to restate this in more representation theoretic language.
\begin{theorem} \label{GoodStuffRepThry}
With notation as above,
\begin{enumerate}
\item[(1)] The $c$-vectors are the dimension vectors of the $V_i$.
\item[(2)]  Reorder the rows and columns of $\tilde{B}$ to match the order of the $V_i$. Let $K_i = \mathrm{End}(V_i)$.
 If $j<k$, then 
$$b^{\vertex}_{jk}=\dim_{K_j} \Ext^1(V_{j},V_{k})  - \dim_{K_j} \Ext^2(V_{j}, V_{k}).$$
If $k<j$ then 
$$b^{\vertex}_{jk} = -\dim_{K_k} \Ext^1(V_{k},V_{j})  + \dim_{K_k} \Ext^2(V_{k}, V_{j}).$$
\end{enumerate}
\end{theorem}

\begin{remark} \label{NoCases}
From the definition of an exceptional sequence, we can restate (2) without cases by writing
\begin{eqnarray*}
b^{\vertex}_{jk}&=&\dim_{K_j} \Ext^1(V_{j},V_{k}) 
- \dim_{K_j} \Ext^2(V_{j}, V_{k}) \\
& &\qquad -\dim_{K_k} \Ext^1(V_{k},V_{j}) 
+ \dim_{K_k} \Ext^2(V_{k}, V_{j}).
\end{eqnarray*}
\end{remark}

\begin{example}
In the situation of Example~\ref{RunningExample}, the $\tilde{B}$ matrix is
$$\left( \begin{array}{r r r r r}
0 & 1 & 0 & -1 & 0 \\
-1 & 0 & 1 & 1 & -1 \\
0 & -1 & 0 & 0 & 1 \\
1 & -1 & 0 & 0 & 0 \\
0 & 1 & -1 & 0 & 0 \\
\hline
1 & 0 & 0 & 0 & 0 \\
0 & 1 & 0 & -1 & 0 \\
0 & 1 & 0 & 0 & 0 \\
0 & 0 & 1 & 0 & -1 \\
0 & 0 & 1 & 0 & 0 \\
\end{array} \right)$$
The top half is computed from the table of $\Ext$ groups in Example~\ref{RunningExample}; the bottom half is computed from the dimension vectors of the exceptional objects.
\end{example}

For many other consequences of the framework result, including formulas for $g$-vectors, see~\cite{RS3}.

\section{Proof of Theorem~\ref{characterize}}

In this section, we 
prove Theorem
\ref{characterize}, the combinatorial characterization of the collections
of $c$-vectors.  


Assume that $b_{ij} \geq 0$ for $i \geq j$.
Let $c$ be the element $s_1 s_2 \cdots s_n$ of $W$, where $s_i$ is the reflection in $\alpha_i$.
We define a Coxeter factorization to be a sequence $(t_1, t_2, \ldots, t_n)$ of reflections of $W$ such that $t_1 t_2 \cdots t_n = c$.
Given an exceptional sequence $(M_1, M_2, \cdots, M_n)$, let $\beta_j = [M_j]$ and let $t_j$ be the reflection in $\beta_j$.
It is easy to see that $t_1 t_2 \cdots t_n$ is a Coxeter factorization; because this property can be showed to be preserved by mutations, and the braid group action on exceptional sequences is well known to be transitive.
Igusa and Schiffler~\cite{IS} showed that, conversely,
given any Coxeter factorization $t_1t_2\dots t_n=c$, there is an
exceptional sequence $(M_1,\dots,M_n)$ such that $t_i$ is the reflection
in $[M_i]$.


\begin{proof}[of Theorem \ref{characterize}]
The fact that the conditions given in Theorem~\ref{characterize} are 
necessary is straightforward.
The exceptional sequence gives rise to a Coxeter factorization.
If $v_i$ and $v_j$ are both positive, then $M_i$ and $M_j$ are objects in $\S$. 
Combining the noncrossing condition with the hereditary nature of $\S$, we see that $\Ext^r(M_i, M_j) = \Ext^r(M_j, M_i) = 0$ for $r \neq 1$, and we deduce that $(v_i, v_j) \leq 0$. 
Similarly, if $v_i$ and $v_j$ are negative than $(v_i, v_j) \leq 0$.

Now suppose that we have a sequence of roots satisfying the conditions
of the theorem.  Let $v_1,\dots,v_r$ be positive roots, and $v_{r+1},
\dots, v_n$ be negative roots, such that the product of the 
corresponding sequence of 
reflections is $s_1\dots s_n$.  
By the main result of Igusa and Schiffler \cite{IS}, 
there is a corresponding exceptional sequence 
$E_1$, \dots, $E_n$ such that $[E_i]$ is the reflection in $v_i$.
For an arbitrary such $E_i$, we have $[E_i] = \pm v_i$. 
By replacing the $E_i$ by appropriate shifts, we may assume that $[E_i] = v_i$ and $E_i \in \S \cup \S[-1]$; we make this assumption from now on.

We now check that this exceptional sequence is noncrossing.  
For $1 \leq i < j \leq r$, condition~(2) tells us that 
$$0 \geq ([E_i], [E_j]) = \dim_{\kk} \Hom(E_i, E_j) - \dim_{\kk} \Ext^1(E_i, E_j)$$
where we have used that $E_i$, $E_j$ is exceptional and that $\S$ is hereditary to remove the other terms defining the symmetric bilinear form.
By Lemma~\ref{onlyone}, at most one of the two terms on the right is nonzero, so it must be the second one. We have shown that $\Hom(E_i, E_j)$ vanishes as desired.
The same argument applies when $r+1\leq i<j\leq n$.  

If $1\leq i\leq r<j\leq n$, so that $E_i \in \S$ and $E_j \in \S[-1]$, then we automatically have
$\Hom(E_i,E_j)=0=\Ext^{-1}(E_i,E_j)$.  This shows that the exceptional
sequence is noncrossing.  By Theorem \ref{repdesc}, it corresponds to some
vertex $v$ in $\Tree$, and we are done.  
\qed\end{proof}

\section{Link to the cluster category}\label{linksec}
Our paper establishes a link between acyclic cluster algebras and the representation theory
of finite-dimensional algebras.  There is, of course, another such link
which is already well-known, going through the construction of cluster categories \cite{bmrrt}.  We will now recall the 
cluster
category in more detail, and explain the connection between these two
categorifications.  

Let $B^0$ be a skew-symmetric, acyclic matrix.  
For this section, we take $\S$ to be the modules over $\kk Q$, with
$Q$ the quiver with $b_{ij}$ arrows from $i$ to $j$, and $\kk$ an
algebraically closed ground field.  

The cluster category
associated to $B^0$ is by definition $\mathcal C=D^b(\S)/[1]\tau^{-1}$, where 
$\tau$ is the Auslander-Reiten translation.  An object $X$ in 
$\mathcal C$ is called {\it exceptional} if it is indecomposable and
satisfies $\Ext^1(X,X)=0$.  An object $T$ is 
{\it cluster tilting} if it is the direct sum of $n$ distinct 
exceptional summands and $\Ext^1(T,T)=0$.

There is a bijection $\phi$ from the cluster variables of the cluster
algebra $\A$ associated to $B^0$ to exceptional
objects of $\mathcal C$, which extends to a bijection from clusters
in $\A$ to cluster tilting objects in $\mathcal C$ \cite{bmrtck,ck2}.
We denote the cluster variables associated to $\vertex \in\Tree$ by 
$\{x_i^\vertex\}$, where the cluster variables are numbered so that if
vertices $\vertex$ and $\vertex'$ are related by an edge labelled $k$, then
$x_i^{\vertex}=x_i^{\vertex'}$ for $i\ne k$.  


%

\begin{theorem}\label{link} 
Let $\vertex$ be a vertex of $\Tree$.  Let $(V_{1},
\dots, V_{n})$ be the noncrossing exceptional
sequence described in Theorem~\ref{repdesc}. Let $\murev(V_1, \ldots, V_n)=(X_{n},\dots,X_{1})$.
Then $X_i=\phi(x_i^\vertex)$.  
\end{theorem}

\begin{proof}
The proof is by induction. We first check that the statement holds for the 
initial cluster.  

$\mu_1\dots\mu_{n-1}(S_1,\dots,S_n)=(X,S_1,\dots,S_{n-1})$, and since 
$(P_n,S_1,\dots,S_{n-1})$ is an exceptional sequence, we must have 
$X=P_n$.  Similarly, $\mu_1\dots\mu_{n-2}(P_n,S_1,\dots,S_{n-1})=
(P_n,P_{n-1},S_1,\dots,S_{n-2})$.  It follows that 
$\murev(S_1,\dots,S_n)=(P_n[1],\dots,P_1[1])$, as desired.  

We then check that if the statement holds for $\vertex\in\Tree$, and 
$\vertex'$ is adjacent to $\vertex$ along an edge labelled $i$, then
it also holds for $\vertex'$.
Suppose we have noncrossing sequences $V_\bullet$ and $V'_\bullet$ associated
to the two vertices, so they are related by a noncrossing mutation.  
It follows that $\murev(V_\bullet)$ and $\murev(V'_\bullet)$ differ by 
commutation moves and braiding a single object behind, which implies that
these sequences, viewed as cluster tilting objects, differ in exactly 
one summand, corresponding to the cluster variable being mutated as we
pass between $\vertex$ and $\vertex'$.  
\qed\end{proof}

This leads to some corollaries.  We use essentially none of
the results that have been developed 
about cluster tilting objects in our proof of Theorem 
\ref{link}.  One could therefore use Theorem \ref{link} to redevelop the 
theory of cluster categories. (For example, one could reprove 
that the Gabriel quiver of the cluster
tilting object associated to $\vertex$ encodes $B^\vertex$, and 
that if $T$ and $T'$ are cluster
tilting objects related by mutation, their Gabriel quivers are related by
Fomin-Zelevinsky mutation \cite{BMR}).

\subsection*{Acknowledgements} The authors would like to thank Andrei
Zelevinsky for helpful comments and encouragement. We would also like to thank Nathan Reading for attempting to fit the notations of his joint work with DES as closely as possible to those in this paper, for his patience with the delays that caused, and for helpful comments and questions.  

 During some of the time this work was done, DES was supported by a Clay Research Fellowship; HT is partially 
supported by an NSERC Discovery Grant.  The authors began their collaboration at the International Conference on Cluster Algebras and Related Topics, hosted by IMUNAM; the authors are grateful for the superb opportunities for discussion we found there. Much of HT's work on this paper was
done during a visit to the Hausdorff Institute; he is grateful for the stimulating research conditions which it provided.

\section*{Appendix: Derived Categories of Hereditary Categories}
\renewcommand{\thetheorem}{A.\arabic{theorem}}
\renewcommand{\theremark}{A\arabic{theorem}}
\def\C{\mathcal{C}}
\def\from{\leftarrow}
This paper uses the language of derived categories, because it is the simplest and most natural language in which to present our results.
However, we fear that this might frighten away some readers, who feel that nothing which mentions the word ``derived" can be elementary.
We therefore seek to explain why, in this case, the derived category is not an object to be feared.

Let $A$ be a ring (not necessarily commutative) and let $\C$ be the category of finitely generated $A$-modules. 
We will write $\Hom_{\C}$ and $\Ext_{\C}$ for Hom and Ext of $A$-modules, so that undecorated $\Hom$ and $\Ext$ can stand for the Hom and Ext in the derived category, as they do throughout this paper.
A complex of $A$-modules is a doubly-infinite sequence $\cdots \from C_{-1} \from C_0 \from  C_1 \from C_2 \from  \cdots$ of $A$-modules and $A$-module maps, such that the composition $C_i \from C_{i+1} \from C_{i+2}$ is $0$ for all $i$.
All our complexes will be bounded, meaning that all but finitely many $C_i$ are zero; we usually will not mention this explicitly.
For a complex $C_{\bullet}$, we write $H_i(C_{\bullet})$ for the homology group $\mathrm{Ker}(C_{i-1} \from C_{i})/\mathrm{Im}(C_{i} \from C_{i+1})$.

Objects of the derived category are bounded complexes, but many different bounded complexes can be isomorphic to each other in the derived category and, as usual in category theory, there will be little reason to distinguish isomorphic objects.
For a general derived category, if complexes $B_{\bullet}$ and $C_{\bullet}$ are isomorphic, then we can deduce that $H_i(B_{\bullet}) \cong H_i(C_{\bullet})$, but the converse does not hold.

However, now suppose that the ring $A$ is what is called \emph{hereditary}, meaning that $\Ext_\C^j(M,N)$ vanishes for all $j \geq 2$ and all $A$-modules $M$ and $N$.
Then we have
\begin{theorem}[{\cite[Section I.5.2]{Hap}}]\label{HereditaryEasy}
If $A$ is hereditary, then the complexes $B_{\bullet}$ and $C_{\bullet}$ are isomorphic in the derived category if and only if $H_i(B_{\bullet}) \cong H_i(C_{\bullet})$ for all $i$.
\end{theorem}

\begin{remark} Happel has a standing assumption that $k$ is algebraically closed in the section we cite.
As Happel says, this assumption is ``not really needed", and the careful reader should have little difficulty removing it. \end{remark}

In particular, $C_{\bullet}$ is isomorphic to the complex which has $H_i(C_{\bullet})$ in position $i$, and where all the maps are zero.
If you like, whenever we speak of an object of the derived category, you can use this trick to simply think of a sequence of modules, taking all the maps between them to be zero.  We will generally only be interested in indecomposable
objects in the derived category. If we view an 
indecomposable object as a sequence of modules in this way, exactly one
of the modules in the sequence will be non-zero.  

We introduce the following notations: For an $A$-module $M$, the object $M[i]$ is the complex which is $M$ in position $i$, and $0$ in every other position.
More generally, for any complex $C_{\bullet}$, the complex $C[i]_\bullet$ has $C[i]_j = C_{j-i}$, with correspondingly shifted maps.
We define direct sums of complexes in the obvious way, so $\bigoplus M_i[i]$ is the complex which is $M_i$ in position $i$, with all the maps being $0$.

In a category, one wishes to know the homorphisms between objects, and how to compose them. 
In the derived category, for $M, N$ objects of $\C$, 
we have $\Hom(M[a], N[b]) = 0$ if $a>b$ and $=\Ext_{\C}^{b-a}(M,N)$ if $b \geq a$. 
We sometimes adopt the notation $\Ext^j(B_{\bullet}, C_{\bullet})$ as shorthand for $\Hom(B_{\bullet}, C[j]_\bullet)$, for this reason.
The composition $\Hom(M[a], N[b]) \times \Hom(N[b], P[c]) \to \Hom(M[a], P[c])$ is the Yoneda product $\Ext_{\C}^{b-a}(M,N) \times \Ext_{\C}^{c-b}(N,P) \to \Ext_{\C}^{c-a}(M,P)$.

We have now described morphisms between complexes that have only one nonzero term. 
More generally, let $M_{\bullet} = \bigoplus M_i[i]$ and $N_{\bullet} = \bigoplus N_i[i]$ be two complexes with all maps $0$, then $\Hom(M,N) = \bigoplus_{i,j} \Hom(M_i[i], N_j[j])$.
Given three such complexes $M_{\bullet}$, $N_{\bullet}$ and $P_{\bullet}$, the composition $\Hom(M_{\bullet}, N_{\bullet}) \times \Hom(N_{\bullet}, P_{\bullet}) \to \Hom(M_{\bullet}, P_{\bullet})$ is the sum of the compositions of the individual terms.
So, if one only looks at complexes where all maps are zero, one can view the derived category as a convenient notational device for organizing the $\Ext$ groups and the maps between them.
In particular, when $A$ is hereditary, we really can understand all the objects and morphisms in the derived category in this way.

Finally, we must describe the ``triangles".
This means that, for every map $M_{\bullet} \stackrel{\phi}{\to} N_{\bullet}$, we must construct a complex $E_{\bullet}$ with maps $N_{\bullet} \to E_{\bullet}$ and $E_{\bullet} \to M_{\bullet}[1]$. 
We call this ``completing $M_{\bullet} \stackrel{\phi}{\to} N_{\bullet}$ to a triangle".
The sense in which this construction is natural is somewhat subtle, so we will gloss over this. 
We only use the triangle construction in the case that $M_{\bullet}$ and $N_{\bullet}$ are of the forms $M[a]$ and $N[b]$ for some $A$-modules $M$ and $N$, so we will only discuss it in that case. 
Furthermore, we will now restrict ourselves to the case that $A$ is hereditary. So there is a nonzero homorphism $M[a] \to N[b]$ if and only if $b-a$ is $0$ or $1$.  For notational simplicity we will restrict to the case $a=0$.  

The following theorem is the result of unwinding the definition of a triangle, the relation between $\Hom(M, N[1])$ and extensions between $N$ and $M$, and using Theorem~\ref{HereditaryEasy} to identify a complex with its cohomology. 

\begin{theorem}
Let $A$ be hereditary and let $M$ and $N$ be $A$-modules.

Let $\psi$ an $A$-module map $M \to N$ and $\phi$  the corresponding map $M \to N$ in the derived category.
If $\psi$ is injective then the completion of $M \stackrel{\phi}{\to} N$ to a triangle is isomorphic to $C$ where $C:=\mathrm{Coker}(\psi)$. 
The map $N \to C$ is the tautological projection and the map $C \to M[1]$ comes from the class of $0 \to M \to N \to C \to 0$ in $\Ext^1(C,M)$.

If $\psi$ is surjective then the completion of $M \stackrel{\phi}{\to} N$ to a triangle is isomorphic to $K[1]$, where $K:=\mathrm{Ker}(\psi)$. 
The map $K[1] \to M[1]$ is $(-1)$ times the tautological inclusion and the map $N \to K[1]$ comes from the class of $0 \to K \to M \to N \to 0$ in $\Ext^1(N,K)$.

Let $\psi$ be a class in $\Ext^1(M,N)$ and let $\phi$ be the corresponding map $M \to N[1]$. Then the completion of $M \stackrel{\phi}{\to} N[1]$ to a triangle is isomorphic to $E[1]$, where $E$ is the extension $0 \to N \to E \to M \to 0$ corresponding to $\phi$. The maps $N[1]$ to $E[1]$ and $E[1] \to M[1]$ are 
$(-1)$ times the maps from the extension short exact sequence.
\end{theorem}

\begin{remark} We use the construction of completing to a triangle to define mutation of exceptional sequences. One of the surprising consequences of the theory of exceptional sequences is that all the maps we will deal with are either injective or surjective, so we do not need to know how to complete $\psi: M \to N$ to a triangle if $\psi$ is neither injective nor surjective.
For the interested reader, we explain nonetheless. Let $K$, $I$ and $C$ be the kernel, image and cokernel of $\psi$. 
The completion of $M \stackrel{\psi}{\to} N$ to a triangle is noncanonically isomorphic to $C \oplus K[1]$. 
The maps $K[1] \to M[1]$ and $N \to C$ are the tautological maps, the former
multiplied by $-1$.
The maps $C \to M[1]$ and $N \to K[1]$ come from classes in $\Ext^1(C,M)$ and $\Ext^1(N,K)$. The precise classes depend on the noncanonical choice of isomorphism, but one can say that their images in $\Ext^1(C,I)$ and $\Ext^1(I,K)$ correspond to the extensions $0 \to I \to N \to C \to 0$ and $0 \to K \to M \to I \to 0$, respectively.
\end{remark}



\begin{thebibliography}{BMRT}
\bibitem[BMR]{BMR} A. B. Buan, R. Marsh, and I. Reiten,  Cluster mutation via quiver 
representations, \emph{Comment. Math. Helv.} \textbf{83} (2008), no. 1, 143--177. 
\bibitem[BMRRT]{bmrrt} A. B. Buan, R. Marsh, M. Reineke, I. Reiten, and G. Todorov, Tilting theory and cluster combinatorics, \emph{Adv. Math.} \textbf{204} (2006), no. 2, 572--618.
\bibitem[BMRTCK]{bmrtck} A. B. Buan, R. Marsh, I. Reiten, and G. Todorov,
Clusters and seeds in acyclic cluster algebras, with an appendix jointly authored with P. Caldero and B. Keller, \emph{Proc. Amer. Math. Soc.} \textbf{135} (2007), no. 10, 3049--3060.
\bibitem[BRT1]{mmm} A. B. Buan, I. Reiten, and H. Thomas, Three kinds of mutations, \emph{J. Algebra} \textbf{339} (2011), 97--113.
\bibitem[BRT2]{art} A. B. Buan, I. Reiten, and H. Thomas,
From m-clusters to m-noncrossing partitions via exceptional sequences,
\emph{Math. Z.} to appear, {\tt arXiv:1007.0928}.   
\bibitem[CK]{ck2} P. Caldero and B. Keller, From triangulated categories to cluster algebras. II, \emph{Ann. Sci. \'Ecole Norm. Sup. (4)} \textbf{39} (2006), no. 6, 983--1009.
\bibitem[DDPW]{DDPW} B. Deng, J. Du, B. Parshall and J. Wang, \emph{Finite dimensional algebras and quantum groups}, Mathematical Surveys and Monographs \textbf{150}, Amer. Math. Soc. 2008.
\bibitem[Dem]{dem} L. Demonet,   Mutations of group species with potentials
and their representations.  Application to cluster algebras,  \texttt{arXiv:1003.5078}.
\bibitem[DWZ]{DWZ} H. Derksen, J. Weyman, and A. Zelevinsky,  Quivers with potentials
and their representations. II, \emph{J. Amer. Math. Soc.} \textbf{23} (2010), no. 3, 749--790.  
\bibitem[FZ]{CA4} S. Fomin and A. Zelevinsky,  Cluster algebras IV.  Coefficients, \emph{Compos. Math.} \textbf{143} (2007), no. 1, 112--164.   
\bibitem[FK]{FK} C. Fu and B. Keller, On cluster algebras with coefficients and $2$-Calabi-Yau categories,
\emph{Trans. Amer. Math. Soc.} \textbf{362} (2010), no. 2, 859--895.
\bibitem[GK]{GK} A. Gorodentsev and S. Kuleshov,  Helix theory, 
{\it Mosc. Math. J.} \textbf{4} (2004), no. 2, 377--"440, 535. 
\bibitem[Hap]{Hap} D. Happel, \emph{Triangulated categories in the representation theory of finite dimensional algebras}, London Mathematical Society Lecture Note Series, 119. Cambridge University Press, Cambridge, 1988.
\bibitem[HR]{hr} D. Happel and C. M. Ringel, Tilted algebras, {\it Trans. Amer. Math. Soc.} \textbf{274} (1982), no. 2, 399--443.
\bibitem[Hub]{hub} A. Hubery,  The cluster complex of an hereditary Artin algebra, Algebras and Representation Theory, \textbf{14} (6), 
1163--1185.  
\bibitem[IS]{IS} K. Igusa and R. Schiffler, Exceptional sequences and clusters, \emph{J. Algebra} \textbf{323} (2010), no. 8, 2183--2202.
\bibitem[KQ]{KQ} A. King and Y. Qiu, Exchange graphs of acyclic Calabi-Yau categories, \texttt{arXiv:1109.2924}.
\bibitem[Nag]{Nag} K. Nagao, Donaldson-Thomas theory and cluster algebras, \texttt{arXiv:1002.4884}
\bibitem[NZ]{NZ} T. Nakanishi and A. Zelevinsky, On tropical dualities in cluster algebras, \texttt{arXiv:1101.3736}.
\bibitem[Pla]{Pla} P.-G. Plamondon, Cluster algebras via cluster categories with infinite-dimensional morphism spaces, \texttt{arXiv:1004.0830}.  
\bibitem[RS1]{RS1} N. Reading and D. Speyer, Cambrian fans, \emph{J. Eur. Math. Soc.} \textbf{11} (2009), no. 2, 407--447.
\bibitem[RS2]{RS2} N. Reading and D. Speyer, Sortable elements for quivers with cycles. \emph{Electron. J. Combin.} \textbf{17} (2010), no. 1.
\bibitem[RS3]{RS3} N. Reading and D. Speyer, Combinatorial frameworks for cluster algebras, \texttt{arXiv:1111.2652}.
\bibitem[Ri1]{rin1} C. M. Ringel, Representations of $K$-species and bimodules, 
J. Algebra 41 (1976), no. 2, 269--302.
\bibitem[Ri2]{rin2} C. M. Ringel, The braid group action on the set of exceptional sequences of a hereditary Artin algebra, Abelian group theory and related topics (Oberwolfach, 1993), 339--352, Contemp. Math., 171, Amer. Math. Soc., 
Providence, RI, 1994.
\end{thebibliography}
\end{document}